\documentclass[12pt]{amsart}
\usepackage{amsfonts,amssymb,amsthm,amsmath,graphicx, hyperref}
\newtheorem{thm}{Theorem}[section]

\newtheorem{lemma}[thm]{Lemma}

\newcommand{\Z}{\mathbb{Z}}
\newcommand{\SL}{{\text {\rm SL}}}

\begin{document}
\title{Weakly holomorphic modular forms in prime power levels of genus zero}
\author{Paul Jenkins}
\address{Department of Mathematics, Brigham Young University, Provo, UT 84602, USA}

\author{DJ Thornton}
\thanks{This work was partially supported by a grant from the Simons Foundation (\#281876 to Paul Jenkins).}

\begin{abstract}
Let $M_k^\sharp(N)$ be the space of weight $k$, level $N$ weakly holomorphic modular forms with poles only at the cusp at $\infty$. We explicitly construct a canonical basis for $M_k^\sharp(N)$ for $N\in\{8,9,16,25\}$, and show that many of the Fourier coefficients of the basis elements in $M_0^\sharp(N)$ are divisible by high powers of the prime dividing the level $N$. Additionally, we show that these basis elements satisfy a Zagier duality property, and extend Griffin's results on congruences in level 1 to levels 2, 3, 4, 5, 7, 8, 9, 16, and 25.
\end{abstract}

\maketitle

\section{Introduction}
Let $M_k^!(N)$ be the space of weakly holomorphic modular forms of even integer weight $k$ and level $N$, and let $M_k^\sharp(N)$ be the subspace of $M_k^!(N)$ with poles allowed only at the cusp at $\infty$.  There is a canonical basis for $M_k^\sharp(N)$ consisting of forms $f_{k, m}^{(N)}(z)$ indexed by the order of vanishing $m$  at the cusp at $\infty$.  If $\Gamma_0(N)$ is genus zero and $n_0$ is the maximal order of vanishing at $\infty$ for a modular form in $M_k^\sharp(N)$, we define the form $f_{k, m}^{(N)}(z)$ to be the unique form in $M_k^\sharp(N)$ with Fourier expansion $f_{k, m}^{(N)}(z) = q^{-m} + \sum_{n > n_0} a_k^{(N)}(m, n) q^n$, where $q = e^{2\pi i z}$ as usual, so that the gap in the Fourier expansion between the $q^{-m}$ term and the next nonzero coefficient is as large as possible.  Such bases were explicitly constructed for levels 1, 2, 3, 4, 5, 7, 13 in~\cite{Duke, Sharon, Haddock, JT}; in all of these cases, the coefficients $a_k^{(N)}(m, n)$ are integers.

For the classical $j$-function $j(z) = \frac{E_4^3(z)}{\Delta(z)} = q^{-1} + 744 + \sum_{n=1}^\infty c(n) q^n \in M_0^\sharp(1)$, Lehner~\cite{Lehner1, Lehner2} proved in 1949 that the coefficients $c(n)$ satisfy the congruences \[c(2^a3^b5^c7^d) \equiv 0 \pmod{2^{3a+8}3^{2b+3}5^{c+1}7^d}\] for nonnegative integers $a, b, c, d$.  These divisibility results were refined into congruences modulo larger powers of small primes by Kolberg~\cite{Kolberg1, Kolberg2} and Aas~\cite{Aas}, and these congruences were extended to the Fourier coefficients $a_0(m, n)$ of all elements $f_{0, m}$ of the canonical basis for $M_0^!(1)$ by Griffin~\cite{Griffin}.  For $N = 2, 3, 4, 5, 7, 13$, divisibility results similar to Lehner's theorem modulo powers of $N$ were proved in~\cite{Haddock, JT} for the coefficients $a_0^{(N)}(m, n)$ of the basis elements $f_0^{(N)}(z)$.  In this note, for the genus zero prime power levels $N = 8, 9, 16, 25$ we prove analogous congruences for the Fourier coefficients $a_0^{(N)}(m, n)$ of basis elements of weight 0.  Additionally, we give results on duality of Fourier coefficients and generating functions for these canonical bases, similar to known results in lower levels, and extend some of Griffin's results on congruences in level 1 to levels 2, 3, 4, 5, 7, 8, 9, 16, and 25.

\section{Previous work}

Generalizing Lehner's theorem on the coefficients of $j(z)$ to the genus zero prime levels $p = 2, 3, 5, 7$, the first author and Andersen~\cite{Nick} proved the following divisibility results for the Fourier coefficients $a_0^{(p)}(m, n)$ when the power of $p$ dividing the exponent $n$ is greater than the power of $p$ dividing the order $m$ of the pole.
\begin{thm}[Theorem 2~\cite{Nick}]\label{thm:Nick}
Let $p \in\{2,3,5,7\}$, and let
\[f_{0,m}^{(p)}(z) = q^{-m}+\sum\limits_{n=1}^\infty a_0^{(p)}(m,n)q^n\]
be an element of the canonical basis for $M_0^\sharp(p)$, with $m=p^\alpha m'$, $n=p^\beta n'$ and $(m',p)=(n',p)=1$. Then for $\beta > \alpha$, we have
\begin{align*}
a_0^{(2)}(2^\alpha m',2^\beta n') \equiv 0 & \pmod{2^{3(\beta-\alpha)+8}} & {\rm if}\ p=2, \\
a_0^{(3)}(3^\alpha m',3^\beta n') \equiv 0 & \pmod{3^{2(\beta-\alpha)+3}} & {\rm if}\ p=3, \\
a_0^{(5)}(5^\alpha m',5^\beta n') \equiv 0 & \pmod{5^{(\beta-\alpha)+1}} & {\rm if}\ p=5, \\
a_0^{(7)}(7^\alpha m',7^\beta n') \equiv 0 & \pmod{7^{(\beta-\alpha)}} & {\rm if}\ p=7.
\end{align*}
\end{thm}
Using this theorem, the authors~\cite{JT} proved the following divisibility results for the cases with $\alpha > \beta$.
\begin{thm}[Theorem 1~\cite{JT}]\label{thm:Thornton1} Let $p\in\{2,3,5,7,13\}$ and let \[f_{0,m}^{(p)}(z) = q^{-m} + \sum_{n=1}^\infty a_0^{(p)}(m,n)q^n\] be an element of the canonical basis for $M_0^\sharp(p)$, with $m=p^\alpha m'$, $n=p^\beta n'$ and $(m',p)=(n',p)=1$. Then for $\alpha>\beta$, we have
\begin{align*}
a_0^{(2)}(2^\alpha m',2^\beta n') \equiv 0 & \pmod{2^{4(\alpha-\beta)+8}} & {\rm if}\ p=2, \\
a_0^{(3)}(3^\alpha m',3^\beta n') \equiv 0 & \pmod{3^{3(\alpha-\beta)+3}} & {\rm if}\ p=3, \\
a_0^{(5)}(5^\alpha m',5^\beta n') \equiv 0 & \pmod{5^{2(\alpha-\beta)+1}} & {\rm if}\ p=5, \\
a_0^{(7)}(7^\alpha m',7^\beta n') \equiv 0 & \pmod{7^{2(\alpha-\beta)}} & {\rm if}\ p=7, \\
a_0^{(13)}(13^\alpha m',13^\beta n') \equiv 0 & \pmod{13^{\alpha-\beta}} & {\rm if}\ p = 13.
\end{align*}
\end{thm}

Additionally, for $N=4$, the authors proved the following theorem.
\begin{thm}[Theorem 2~\cite{JT}]\label{thm:Thornton3}
Let $f_{0,m}^{(4)}(z) = q^{-m} + \sum a_0^{(4)}(m,n)q^n \in M_0^\sharp(4)$ be an element of the canonical basis. Write $m = 2^\alpha m'$ and $n = 2^\beta n'$ with $m',n'$ odd. Then for $\alpha\neq \beta$, we have
\begin{align*}
a_0^{(4)}(2^\alpha m',2^\beta n') \equiv 0 & \pmod{2^{4(\alpha-\beta)+8}} & {\rm if}\ \alpha > \beta, \\
a_0^{(4)}(2^\alpha m',2^\beta n') \equiv 0 & \pmod{2^{3(\beta-\alpha)+8}} & {\rm if}\ \beta > \alpha.
\end{align*}
\end{thm}

In~\cite{Griffin}, Griffin proved the following congruence results for the Fourier coefficients $a_0^{(1)}(m, n)$ of the canonical basis elements for $M_0^{!}(1) = M_0^\sharp(1)$.
\begin{thm}[Theorem 2.1 \cite{Griffin}]\label{thm:Griffin}
Write $m = p^\alpha m'$ and $p^\beta n'$ where $\alpha,\beta\geq 0$ and $m',n' \not\equiv 0\pmod{p}$. Then the following congruences hold.

\noindent For $p=2$:
\begin{align*}
    a_0^{(1)}(2^\alpha m', 2^\beta n') &\equiv -2^{3(\beta-\alpha) + 8}3^{\beta-\alpha-1}\cdot m'\sigma_7(m')\sigma_7(n') \pmod{2^{3(\beta-\alpha)+13}} \hspace{0.5cm} \text{if\ } \beta > \alpha,\\
    &\equiv -2^{4(\beta-\alpha) + 8}3^{\beta-\alpha-1}\cdot m'\sigma_7(m')\sigma_7(n') \pmod{2^{4(\beta-\alpha)+13}} \hspace{0.5cm} \text{if\ } \alpha > \beta, \\
    &\equiv 20m'\sigma_7(m')\sigma_7(n') \pmod{2^7} \hspace{0.5cm} \text{if\ } \alpha=\beta, m'n'\equiv 1\pmod{8}, \\
    &\equiv \frac{1}{2}m'\sigma(m')\sigma(n') \pmod{2^3} \hspace{0.5cm} \text{if\ } \alpha=\beta, m'n' \equiv 3\pmod{8}, \\
    &\equiv -12m'\sigma_7(m')\sigma_7(n') \pmod{2^8} \hspace{0.5cm} \text{if\ } \alpha=\beta, m'n' \equiv 5\pmod{8}.
\end{align*}
For $p=3$:
\begin{align*}
    a_0^{(1)}(3^\alpha m',3^\beta n') &\equiv \mp3^{2(\beta-\alpha)+3}10^{\beta-\alpha-1} \frac{\sigma(m')\sigma(n')}{n'} \pmod{3^{2(\beta-\alpha)+6}} \\& \hspace{2.5cm} \text{if\ } \beta>\alpha, m'n' \equiv \pm1\pmod{3}, \\
    &\equiv \mp3^{3(\alpha-\beta)+3}10^{\alpha-\beta-1} \frac{\sigma(m')\sigma(n')}{n'} \pmod{3^{3(\alpha-\beta)+6}} \\& \hspace{2.5cm} \text{if\ } \alpha>\beta, m'n' \equiv \pm1\pmod{3}, \\
    &\equiv 2\cdot3^3\frac{\sigma(m')\sigma(n')}{n'} \pmod{3^7} \\& \hspace{2.5cm} \text{if\ } \alpha=\beta, m'n' \equiv 1\pmod{3}.
\end{align*}
For $p=5$:
\begin{align*}
    a_0^{(1)}(5^\alpha m',5^\beta n') &\equiv -5^{\beta-\alpha+1}3^{\beta-\alpha-1}(m')^2n'\sigma(m') \sigma(n') \pmod{5^{\beta-\alpha+2}} \hspace{0.5cm} \text{if\ } \beta>\alpha, \\
    &\equiv -5^{2(\alpha-\beta)+1}3^{\alpha-\beta-1}(m')^2n'\sigma(m')\sigma(n') \pmod{5^{2(\alpha-\beta)+2}} \hspace{0.5cm} \text{if\ } \alpha>\beta, \\
    &\equiv 10(m')^2n'\sigma(m')\sigma(n') \pmod{5^2} \hspace{0.5cm} \text{if\ } \alpha=\beta, \left(\frac{m'n'}{5}\right) = -1.
\end{align*}
For $p=7$:
\begin{align*}
    a_0^{(1)}(7^\alpha m',7^\beta n') &\equiv 7^{\beta-\alpha}5^{\beta-\alpha-1}(m')^2n'\sigma_3(m') \sigma_3(n') \pmod{7^{\beta-\alpha+1}} \hspace{0.5cm} \text{if\ } \beta>\alpha, \\
    &\equiv 7^{2(\alpha-\beta)}5^{\alpha-\beta-1}(m')^2n'\sigma_3(m')\sigma_3(n') \pmod{7^{2(\alpha-\beta)+1}} \hspace{0.5cm} \text{if\ } \alpha>\beta, \\
    &\equiv 2(m')^2n'\sigma_3(m')\sigma_3(n') \pmod{7} \hspace{0.5cm} \text{if\ } \alpha=\beta, \left(\frac{m'n'}{7}\right) = 1.
\end{align*}
\end{thm}

\section{Notation and statement of results}

The first main result of this paper is the following theorem, extending Theorems~\ref{thm:Nick} and~\ref{thm:Thornton1} and giving divisibility results for the coefficients $a_0^{(N)}(m, n)$ in the prime power levels $N = 8, 9, 16, 25$.
\begin{thm}\label{Congruences}
Let $f_{0, m}^{(N)}(z) \in M_0^\sharp(N)$ be an element of the canonical basis for $N \in \{8, 9, 16, 25\}$.  Assume that $(m', N) = (n', N) = 1$.  The following congruences hold for $\alpha \neq \beta$.
\begin{align*}
a_0^{(8)}(2^\alpha m', 2^\beta n')  \equiv a_0^{(16)}(2^\alpha m', 2^\beta n')  \equiv & 0  \pmod{2^{4(\alpha-\beta)+8}} & {\rm if}\ \alpha > \beta, \\
a_0^{(8)}(2^\alpha m',2^\beta n')  \equiv a_0^{(16)}(2^\alpha m', 2^\beta n')  \equiv & 0  \pmod{2^{3(\beta-\alpha)+8}} & {\rm if}\ \beta > \alpha, \\
a_0^{(9)}(3^\alpha m', 3^\beta n')  \equiv & 0 \pmod{3^{3(\alpha - \beta) + 3}} & {\rm if}\ \alpha > \beta, \\
a_0^{(9)}(3^\alpha m', 3^\beta n')  \equiv & 0 \pmod{3^{2(\beta - \alpha) + 3}} & {\rm if}\ \beta > \alpha, \\
a_0^{(25)}(5^\alpha m', 5^\beta n')  \equiv & 0 \pmod{5^{2(\alpha - \beta) + 1}} & {\rm if}\ \alpha > \beta, \\
a_0^{(25)}(5^\alpha m', 5^\beta n')  \equiv & 0 \pmod{5^{(\beta - \alpha)+1}} & {\rm if}\ \beta > \alpha.
\end{align*}
\end{thm}

The space $M_k^\sharp(N)$ has a subspace $S_k^\sharp(N)$ consisting of forms which may have poles at the cusp at $\infty$, but which vanish at all other cusps of $\Gamma_0(N)$.  Just as the $f_{k, m}^{(N)}$ form a basis for $M_k^\sharp(N)$, the space $S_k^\sharp(N)$ has a canonical basis $g_{k, m}^{(N)}$ with Fourier expansions of the form \[g_{k, m}^{(N)}(z) = q^{-m} + \sum_{n > n_1} b_k^{(N)}(m, n) q^n,\] where $n_1$ is the maximal order of vanishing at $\infty$ of a form in $S_k^\sharp(N)$.  The second main result of this paper gives a duality result between the coefficients of $f_{k, m}^{(N)}$ and the coefficients of $g_{2-k, n}^{(N)}$.  Similar theorems giving such duality results in levels $1, 2, 3, 4, 5, 7, 13$ appeared in~\cite{Duke, Sharon, Haddock, JT}.
\begin{thm}\label{Duality}
For $N = 8, 9, 16, 25$, any even integer weight $k$, and any integers $m, n$, we have \[a_k^{(N)}(m, n) = -b_{2-k}^{(N)}(n, m).\]
\end{thm}

The third main result of this paper is an extension, in certain cases, of Griffin's results in Theorem~\ref{thm:Griffin} to levels $2, 3, 4, 5, 7, 8, 9, 16$, and $25$.
\begin{thm}\label{Griffinextension}
Let $N \in \{2, 3, 4, 5, 7, 8, 9, 16, 25\}$.  We have the following congruences for the coefficients $a_0^{(N)}(p^\alpha m', p^\beta n')$, where $p$ is the prime dividing $N$ and $(m', p) = (n', p) = 1$.

\begin{align*} a_0^{(2)}(2^\alpha m',2^\beta n') &\equiv -2^{11}m'\sigma_7(m')\sigma_7(n') \pmod{2^{16}}, &\text{if\ } \alpha = \beta -1, \\
&\equiv 20m'\sigma_7(m')\sigma_7(n') \pmod{2^7}, &\text{if\ } \alpha=\beta, m'n'\equiv1\pmod{8}, \\
&\equiv \frac{1}{2}m'\sigma(m')\sigma(n') \pmod{2^3}, &\text{if\ } \alpha=\beta, m'n'\equiv3\pmod{8}, \\
&\equiv -12m'\sigma_7(m')\sigma_7(n') \pmod{2^8}, &\text{if\ } \alpha=\beta, m'n'\equiv5\pmod{8}.
\end{align*}

\begin{align*} a_0^{(3)}(3^\alpha m',3^\beta n') &\equiv \mp 3^5\frac{\sigma(m')\sigma(n')}{n'} \pmod{3^8}, &\text{if\ } \alpha = \beta - 1, m'n'\equiv \pm1\pmod{3}, \\
&\equiv \mp 3^6\frac{\sigma(m')\sigma(n')}{n'} \pmod{3^9}, &\text{if\ } \alpha = \beta + 1, m'n'\equiv \pm1\pmod{3}, \\
&\equiv \frac{2\cdot3^3\sigma(m')\sigma(n')}{n'} \pmod{3^7}, &\text{if\ } \alpha=\beta, m'n' \equiv 1 \pmod{3}.
\end{align*}

\begin{align*}
a_0^{(5)}(5^\alpha m',5^\beta n') &\equiv -5^{\beta-\alpha+1}3^{\beta-\alpha-1}(m')^2n'\sigma(m')\sigma(n') \pmod{5^{\beta-\alpha+2}}, &\text{if\ } 0 < \beta-\alpha \leq 3, \\
&\equiv -5^3(m')^2n'\sigma(m')\sigma(n') \pmod{5^4}, &\text{if\ } \alpha = \beta +1, \\
&\equiv 10(m')^2n'\sigma(m')\sigma(n') \pmod{5^2}, &\text{if\ } \alpha = \beta, \left(\frac{m'n'}{5}\right) = -1. \end{align*}

\begin{align*}
a_0^{(7)}(7^\alpha m',7^\beta n') &\equiv 7^{\beta-\alpha+1}5^{\beta-\alpha-1}(m')^2n'\sigma_3(m')\sigma_3(n') \pmod{7^{\beta-\alpha+1}}, &\text{if\ } 0 < \beta-\alpha \leq 3, \\
&\equiv 7^2(m')^2(n')\sigma_3(m')\sigma_3(n') \pmod{7^3}, &\text{if\ } \alpha=\beta+1, \\
&\equiv 2(m')^2n'\sigma_3(m')\sigma_3(n') \pmod{7}, &\text{if\ } \alpha=\beta, \left(\frac{m'n'}{7}\right)=1. \end{align*}

\begin{align*}
a_0^{(4)}(2^\alpha m',2^\beta n') &\equiv -2^{11}m'\sigma_7(m')\sigma_7(n') \pmod{2^{16}}, &\text{if\ } \alpha = \beta - 1, \\
&\equiv 20m'\sigma_7(m')\sigma_7(n') \pmod{2^7}, &\text{if\ } \alpha = \beta, m'n'\equiv1\pmod{8}, \\
&\equiv \frac{1}{2}m'\sigma(m')\sigma(n') \pmod{2^3},  &\text{if\ } \alpha = \beta, m'n'\equiv3\pmod{8}, \\
&\equiv -12m'\sigma_7(m')\sigma_7(n') \pmod{2^8}, &\text{if\ } \alpha = \beta, m'n'\equiv5\pmod{8}.
\end{align*}

\begin{align*}
a_0^{(8)}(2^\alpha m',2^\beta n') &\equiv -2^{11}m'\sigma_7(m')\sigma_7(n') \pmod{2^{16}} &\text{if\ } \alpha = \beta - 1, \\
&\equiv 20m'\sigma_7(m')\sigma_7(n') \pmod{2^7}, &\text{if\ } \alpha = \beta \neq 0, m'n'\equiv1\pmod{8}, \\
&\equiv \frac{1}{2}m'\sigma(m')\sigma(n') \pmod{2^3}, &\text{if\ } \alpha = \beta \neq 0. m'n'\equiv3\pmod{8}, \\
&\equiv -12m'\sigma_7(m')\sigma_7(n') \pmod{2^8}, &\text{if\ } \alpha = \beta \neq 0, m'n'\equiv5\pmod{8}, \\
&\equiv \frac{1}{2}m'\sigma(m')\sigma(n') \pmod{2^3}, &\text{if\ } \alpha=\beta=0, m'n'\equiv3\pmod{8}.
\end{align*}

\begin{align*}
a_0^{(16)}(2^\alpha m',2^\beta n') &\equiv -2^{11}m'\sigma_7(m')\sigma_7(n') \pmod{2^{16}}, &\text{if\ } \alpha=\beta-1, \\
&\equiv 20m'\sigma_7(m')\sigma_7(n') \pmod{2^7}, &\text{if\ } \alpha=\beta > 1, m'n'\equiv1\pmod{8}, \\
&\equiv \frac{1}{2}m'\sigma(m')\sigma(n') \pmod{2^3},  &\text{if\ } \alpha=\beta > 1, m'n'\equiv3\pmod{8}, \\
&\equiv -12m'\sigma_7(m')\sigma_7(n') \pmod{2^8}, &\text{if\ } \alpha=\beta > 1, m'n'\equiv5\pmod{8}, \\
&\equiv \frac{1}{2}m'\sigma(m')\sigma(n') \pmod{2^3}, &\text{if\ } \alpha=\beta=1, m'n'\equiv3\pmod{8}.
\end{align*}

\begin{align*}
a_0^{(9)}(3^\alpha m',3^\beta n') &\equiv \mp 3^5\frac{\sigma(m')\sigma(n')}{n'} \pmod{3^8}, &\text{if\ } \alpha = \beta - 1, m'n'\equiv \pm1\pmod{3}, \\
&\equiv \mp 3^6\frac{\sigma(m')\sigma(n')}{n'} \pmod{3^9}, &\text{if\ } \alpha = \beta + 1, m'n'\equiv \pm1\pmod{3}, \\
&\equiv \frac{2\cdot3^3\sigma(m')\sigma(n')}{n'} \pmod{3^7}, &\text{if\ } \alpha = \beta, m'n'\equiv 1\pmod{3}.
\end{align*}

\begin{align*}
a_0^{(25)}(5^\alpha m',5^\beta n') &\equiv -5^{\beta-\alpha+1}3^{\beta-\alpha-1}(m')^2n'\sigma(m')\sigma(n') \pmod{5^{\beta-\alpha+2}}, &\text{if\ } 0 < \beta-\alpha \leq 3, \\
&\equiv -5^3(m')^2n'\sigma(m')\sigma(n') \pmod{5^4}, &\text{if\ } \alpha = \beta +1, \\
&\equiv 10(m')^2n'\sigma(m')\sigma(n') \pmod{5^2}, &\text{if\ } \alpha = \beta, \left(\frac{m'n'}{5}\right) = -1.
\end{align*}

\end{thm}

\section{Constructing canonical bases}

In this section, we explicitly construct canonical bases for the spaces $M_k^\sharp(N)$ and $S_k^\sharp(N)$ for the levels $N = 8, 9, 16, 25$ and prove Theorem~\ref{Duality}.

The congruence subgroup $\Gamma_0(8)$ has 4 cusps, which may be taken to be at $0, \frac{1}{2}, \frac{1}{4}, \infty$. We choose the Hauptmodul \[\psi^{(8)}(z) = \frac{\eta^4(z)\eta^2(4z)}{\eta^2(2z)\eta^4(8z)} = q^{-1} -4+4q+2q^3+\cdots \in M_0^\sharp(8),\] where $\eta(z) = q^{1/24}\prod_{n=1}^\infty (1-q^n)$ is the Dedekind eta function.  The values of $\psi^{(8)}(z)$ at the cusps $0, \frac{1}{2}, \frac{1}{4}$ are $0, -8, -4$ respectively.  The basis elements $f_{0, m}^{(8)}(z)$ of weight $0$ for $m \geq 0$ are then polynomials of degree $m$ in $\psi^{(8)}(z)$, chosen so that the Fourier expansion begins \[f_{0, m}^{(8)}(z) = q^{-m} + \sum_{n \geq 1} a_0(m, n)q^n.\]

For spaces $M_k^\sharp(8)$ of arbitrary even weight $k$, the first basis element of weight $k$ may be obtained from the first basis element of weight $k+2$ or weight $k-2$ by multiplying or dividing by the form \[S^{(8)}(z) = \frac{\eta^8(8z)}{\eta^4(4z)} = q^2+4q^6+6q^{10}+\cdots \in M_2(8),\] which has all of its zeros at the cusp at $\infty$.  Thus, the first basis element will always be $(S^{(8)}(z))^{k/2}$.  The rest of the basis elements in weight $k$ may be constructed by multiplying by powers of $\psi^{(8)}(z)$ to get the appropriate leading term $q^{-m}$ and subtracting earlier basis elements to obtain the gap in the Fourier expansion.  We thus construct canonical bases for all $M_k^\sharp(8)$ consisting of modular forms $f_{k, m}^{(8)}(z)$ for each $m \geq -k$ with Fourier expansion
\[f_{k,m}^{(8)}(z) = q^{-m} + \sum_{n= k+1}^\infty a_k^{(8)}(m,n)q^n,\] and all of the coefficients $a_k^{(8)}(m, n)$ are integers.

In level $9$, we use the Hauptmodul $\psi^{(9)}(z) = \frac{\eta^3(z)}{\eta^3(9z)} = q^{-1} - 3 + 5q^2 - 7q^5 + \ldots$, which has a simple pole at $\infty$ and values of $0, 3\sqrt{3}\left(\frac{-\sqrt{3} \mp i}{2}\right)$ at the cusps at $0, \pm \frac{1}{3}$.  Basis elements $f_{0, m}^{(9)}(z)$ are polynomials of degree $m$ in $\psi^{(9)}(z)$, and the first basis element of weight $k \pm 2$ is obtained by multiplying or dividing the first basis element of weight $k$ by \[S^{(9)}(z) = \frac{\eta^6(9z)}{\eta^2(3z)} = q^2 + 2q^5 + 5q^8 +\ldots \in M_2(9).\]

In level $16$, the Hauptmodul is $\psi^{(16)}(z) = \frac{\eta^2(z)\eta(8z)}{\eta(2z)\eta^2(16z)}$, with a pole at $\infty$ and values of $0, -2, -4, -2 \mp 2i$ at the cusps $0, \frac{1}{8}, \frac{1}{2}, \pm \frac{1}{4}$.  The form \[S^{(16)}(z) = \frac{\eta^8(16z)}{\eta^4(8z)} = q^4+4q^{12}+6q^{20}+\ldots \in M_2(16)\] is used to obtain the first basis element in weight $k \pm 2$ from the first basis element in weight $k$.

In level $25$, the Hauptmodul is $\psi^{(25)}(z) = \frac{\eta(z)}{\eta(25z)}$, with a pole at $\infty$ and values of $0$, $\sqrt{5}\left(\frac{1-\sqrt{5}}{4} \mp i\sqrt{\frac{5+\sqrt{5}}{8}}\right)$, $\sqrt{5}\left(\frac{-1-\sqrt{5}}{4} \mp i\sqrt{\frac{5-\sqrt{5}}{8}}\right)$ at the cusps at $0, \pm \frac{1}{5}, \pm \frac{2}{5}$.  We define $E_0^{(25)}(z) = 1$, and note that the first basis element in weight $2$ has $q$-expansion given by \[E_2^{(25)}(z) = f_{2, -4}^{(25)}(z) = q^4 + q^6 + 2q^9 + 3q^{14} + 2q^{16} + \ldots;\] it has maximal order of vanishing at $\infty$ and has integral Fourier coefficients, but is not a linear combination of eta-quotients.  To obtain the first basis element of weight $k \pm 4$ from the first basis element of weight $k$, we multiply or divide by the weight $4$ form given by \[S^{(25)}(z) = \frac{\eta^{10}(25z)}{\eta^2(5z)} = q^{10} + 2q^{15} + 5q^{20} + 10q^{25} + \ldots \in M_4(25),\]
so that if $k = 4\ell + k'$ with $\ell \in \mathbb{Z}$ and $k' \in \{0, 2\}$, then the first basis element in weight $k$ is $E_{k'}^{(25)}(z) (S^{(25)}(z))^\ell$.  To get additional basis elements we multiply earlier basis elements by $\psi^{(25)}(z)$ and row reduce.

Constructing a canonical basis $\{g_{k, m}^{(N)}(z)\}$ for the subspace $S_k^\sharp(N)$ of forms which vanish at all cusps away from $\infty$ follows a similar process.  To obtain the first basis element of $S_k^\sharp(N)$, we multiply the first basis element $f_{k, m}^{(N)}(z)$ of $M_k^\sharp(N)$ by $\prod_\frac{a}{b} (\psi^{(N)}(z) - c_{a/b}),$ where $\frac{a}{b}$ runs over the non-$\infty$ cusps of $\Gamma_0(N)$ and $c_{a/b}$ is the value of $\psi^{(N)}(z)$ at the cusp $\frac{a}{b}$.  Additional basis elements are constructed by multiplying by $\psi^{(N)}(z)$ and row reducing.

If $n_0, n_1$ are the maximal orders of vanishing at $\infty$ for a form in $M_k^\sharp(N), S_k^\sharp(N)$, so that basis elements have the Fourier expansions \[f_{k, m}^{(N)}(z) = q^{-m} + \sum_{n > n_0} a_k^{(N)}(m, n) q^n,\]\[g_{k, m}^{(N)}(z) = q^{-m} + \sum_{n > n_1} b_k^{(N)}(m, n) q^n,\] then Table~\ref{n0n1table} gives values of $n_0$ and $n_1$ for the values of $N$ considered here; note that if $\Gamma_0(N)$ has $c_N$ cusps, then $n_1 = n_0 - c_N+ 1$.  We note also that by construction, all of the coefficients $a_k^{(N)}(m, n)$ are integers for $N = 8, 9, 16, 25$.
\begin{table}[]
\centering
\caption{Values of $n_0, n_1$}
\label{n0n1table}
\def\arraystretch{1.5}
\begin{tabular}{|l|l|l|l|l|}
\hline
$N$   &  8 & 9 &  16 & 25  \\
\hline
$n_0$ &   $k$   & $k$    &$2k$    &  $10\lfloor\frac{k}{4}\rfloor + 2k'$  \\
\hline
$n_1$ &    $k-3$  & $k-3$  &  $2k-5$   & $10\lfloor\frac{k}{4}\rfloor + 2k' - 5$ \\
\hline
\end{tabular}
\end{table}

\begin{proof}[Proof of Theorem~\ref{Duality}] With these bases constructed, we can prove Theorem~\ref{Duality}.  To do so, we note that the product $f_{k, m}^{(N)}(z)g_{2-k, n}^{(N)}(z)$ lies in the space $S_2^\sharp(N)$, and therefore must have a pole at $\infty$ since $n_1 = -1$ for $k=2$ and $N=8, 9, 16, 25$.  By examining the Fourier expansions and the values of $n_0$ and $n_1$, we see that the constant term of this form is $a_k^{(N)}(m, n) + b_{2-k}^{(N)}(n, m)$.

Ramanujan's $\theta$-operator $\theta = \frac{1}{2\pi i} \frac{d}{dz} = q \frac{d}{dq}$ sends modular forms in $M_0^\sharp(N)$ to forms in $S_2^\sharp(N)$, as can be seen from the well-known intertwining relation $\theta(f |_0 \gamma) = (\theta f)|_2 \gamma$ (see for example~\cite{BOR}).  Since $n_1 = -1$, the space $S_2^\sharp(N)$ is spanned by the derivatives $q \frac{d}{dq} f_{0, m}^{(N)}(z)$ for $m \geq 1$, and the constant term above must be zero, proving Theorem~\ref{Duality}. \end{proof}

\section{Proving congruences}

We require the $U_p$ and $V_p$ operators (see~\cite{Atkin}). For a modular form $f(z) = \sum_{n=n_0}^\infty a(n)q^n \in M_k^!(N)$, we define
\begin{align*}
U_p(f(z)) &=  \sum\limits_{n=n_0}^\infty a(pn)q^n \in M_k^{!}(pN), \\
V_p(f(z)) &=  \sum\limits_{n=n_0}^\infty a(n)q^{pn} \in M_k^{!}(pN).
\end{align*}
If $p|N$, then we actually have $U_p(f(z)) \in M_k^!(N)$, while if $p^2 | N$, then $U_p(f(z)) \in M_k^!(N/p)$.

To prove Theorem~\ref{Congruences}, we need the following lemma.
\begin{lemma}\label{ULemma}
The following equalities are true.
\begin{align*}
U_2(f_{0, 2m}^{(8)}(z)) &= f_{0, m}^{(4)}(z), &V_2(f_{0, m}^{(4)}(z)) &= f_{0, 2m}^{(8)}(z) \\
U_3(f_{0, 3m}^{(9)}(z)) &= f_{0, m}^{(3)}(z), &V_3(f_{0, m}^{(3)}(z)) &= f_{0, 3m}^{(9)}(z) \\
U_2(f_{0, 2m}^{(16)}(z)) &= f_{0, m}^{(8)}(z), &V_2(f_{0, m}^{(8)}(z)) &= f_{0, 2m}^{(16)}(z) \\
U_5(f_{0, 5m}^{(25)}(z)) &= f_{0, m}^{(5)}(z) &V_5(f_{0, m}^{(5)}(z)) &= f_{0, 5m}^{(25)}(z).
\end{align*}
Additionally, if $m'$ and the level are relatively prime, we have
\[U_2(f_{0, m'}^{(8)}(z)) = U_3(f_{0, m'}^{(9)}(z)) = U_2(f_{0, m'}^{(16)}(z)) = U_5(f_{0, m'}^{(25)}(z)) = 0.\]
\end{lemma}
From Lemma~\ref{ULemma}, it follows that
\[a_0^{(16)}(4m, 4n) = a_0^{(8)}(2m, 2n) = a_0^{(4)}(m, n),\]
\[a_0^{(9)}(3m, 3n) = a_0^{(3)}(m, n),\]
\[a_0^{(25)}(5m, 5n) = a_0^{(5)}(m, n),\]
and Theorem~\ref{Congruences} follows immediately from Theorems~\ref{thm:Nick}, \ref{thm:Thornton1}, and \ref{thm:Thornton3}.

We next prove Lemma~\ref{ULemma}, beginning with the $N=8$ case.  We first claim that if $f \in M_0^\sharp(8)$, then $V_2U_2f \in M_0^\sharp(8)$.  To see this, note that $f|\gamma$ is holomorphic at $\infty$ for all $\gamma \in \SL_2(\Z)\setminus\Gamma_0(8)$ .  It suffices to show that $(V_2 U_2 f)|\gamma$ is holomorphic at $\infty$ for all such $\gamma$.  However, $V_2 U_2 f(z) = \frac{1}{2} (f(z) + f(z+1/2))$, so it is enough to show that $f|\left[\begin{smallmatrix} 2&1\\0&2 \end{smallmatrix}\right]\left[\gamma\right]$ is holomorphic at $\infty$.  But this can be written as $f|\left[ \alpha \right] \left[ \begin{smallmatrix} \ast&\ast \\ 0&\ast \end{smallmatrix} \right]$ for some $\alpha \in \SL_2(\Z)\setminus\Gamma_0(8)$, so the claim follows.

By looking at Fourier expansions, we see that $V_2 U_2$ acts as the identity on $f_{0, 2m}^{(8)}(z)$ and annihilates $f_{0, 2m+1}^{(8)}(z)$.  Additionally, $V_2(f_{0, m}^{(4)}(z))$ must be a modular form in $M_0^\sharp(8)$ with a pole of order $2m$ at $\infty$, and \[V_2\left(U_2(f_{0,2m}^{(8)}(z)) - f_{0,m}^{(4)}(z)\right)\] is a form in $M_0^\sharp(8)$ which vanishes at $\infty$ and must therefore be identically zero, proving Lemma~\ref{ULemma} for $N=8$.  The proof for $N=9, 16, 25$ is similar.

For $p = 2, 3, 5, 7, 13$ define \[\psi^{(p)}(z) = \left(\frac{\eta(z)}{\eta(pz)}\right)^{\frac{24}{p-1}} \in M_0^\sharp(p), \phi^{(p)}(z) = \frac{1}{\psi^{(p)}(z)}.\]  We note that $\psi^{(p)}(z)$ has a pole at $\infty$ and a zero at $0$ and is equal to $f_{0, 1}^{(p)}(z) - C$ for some constant $C$, and that $\phi^{(p)}(z)$ vanishes at $\infty$ and has a pole at $0$.  To prove Theorem~\ref{Griffinextension}, we will use the following equalities of modular functions, writing $j(z) = f_{0, 1}^{(1)}(z) +744$ in terms of $\psi^{(p)}(z)$ and $\phi^{(p)}(z)$.
\begin{align*}
j(z) = &\psi^{(2)}(z) (1+2^8 \phi^{(2)}(z))^3, \\
     = &\psi^{(3)}(z) + 756 + 196830\phi^{(3)}(z) + 19131876(\phi^{(3)}(z))^2 + 387420489(\phi^{(3)}(z))^3, \\
     = &\psi^{(5)}(z) + 750 + 196875\phi^{(5)}(z) + 20312500(\phi^{(5)}(z))^2 \\
       &+ 615234375(\phi^{(5)}(z))^3 + 7324218750(\phi^{(5)}(z))^4 + 30517578125(\phi^{(5)}(z))^5, \\
     = &\psi^{(7)}(z) + 748 + 196882\phi^{(7)}(z) + 20706224(\phi^{(7)}(z))^2 + 695893835(\phi^{(7)}(z))^3 \\
       &+ 10976181104(\phi^{(7)}(z))^4 + 90957030178(\phi^{(7)}(z))^5 \\
       &+ 38756041628(\phi^{(7)}(z))^6 + 678223072849(\phi^{(7)}(z))^7.
\end{align*}
We use these and similar equalities to compute that
\begin{align*}
f_{0,1}^{(1)}(z) &\equiv f_{0,1}^{(2)}(z) \pmod{2^{16}}, \\
                 &\equiv f_{0,1}^{(3)} \pmod{3^9}, \\
                 &\equiv f_{0,1}^{(5)}(z) \pmod{5^5}, \\
                 &\equiv f_{0,1}^{(7)}(z) \pmod{7^4}, \\
                 &\equiv f_{0,1}^{(4)}(z) \pmod{2^8}, \\
                 &\equiv f_{0,1}^{(8)}(z) \pmod{2^4}, \\
                 &\equiv f_{0,1}^{(16)}(z) \pmod{2^2}, \\
                 &\equiv f_{0,1}^{(9)}(z) \pmod{3^3}, \\
                 &\equiv f_{0,1}^{(25)}(z) \pmod{5}.
\end{align*}

Writing $f_{0, m}^{(1)}(z)$ as a polynomial in $j(z)$ and using $j(z) = \psi^{(2)}(z) (1+2^8 \phi^{(2)}(z))^3$, we find that
\begin{align*}
f_{0, m}^{(1)}(z) &= j(z)^m + a_{m-1}j(z)^{m-1} + \ldots + a_0 \\
                  &= \psi^{(2)}(z)^m + A_{m-1} \psi^{(2)}(z)^{m-1} + \ldots + A_0 + 2^{16}P(\phi^{(2)}(z)),
\end{align*}
where $A_i \in \Z$ and $P(x)$ is a polynomial with integer coefficients.  Similarly, we can write $f_{0, m}^{(2)}(z)$ as a polynomial in $\psi^{(2)}(z)$ with integer coefficients, so that
\[f_{0, m}^{(2)}(z) = \psi^{(2)}(z)^m + B_{m-1}\psi^{(2)}(z)^{m-1} + \ldots + B_0\]
with $B_i \in \Z$.  Subtracting the two polynomials, we obtain
\begin{align*}
f_{0, m}^{(1)}(z) - f_{0, m}^{(2)}(z) &= \sum_{n \geq 1} C_n q^n \\
 &= \left(\sum_{i=0}^m (A_i - B_i) \psi^{(2)}(z)^i \right) + 2^{16}P(\phi^{(2)}(z)).
\end{align*}
Because this is a modular form in $M_0^!(2)$ which vanishes at $\infty$ and therefore cannot have poles of order $i$ at $\infty$, the coefficients $A_i - B_i$ of $\psi^{(2)}(z)^i$ on the right side must be zero for $i = m, m-1, m-2, \ldots, 1, 0$ in turn, and the difference $f_{0, m}^{(1)}(z) - f_{0, m}^{(2)}(z)$ must be divisible by $2^{16}$.  Using similar methods in other levels, we obtain the following result.
\begin{lemma}
For any $m > 0$, we have
\begin{align*}
f_{0, m}^{(1)}(z) &\equiv f_{0, m}^{(2)}(z) \pmod{2^{16}}, \\
                  &\equiv f_{0, m}^{(3)}(z) \pmod{3^9}, \\
                  &\equiv f_{0, m}^{(5)}(z) \pmod{5^5}, \\
                  &\equiv f_{0, m}^{(7)}(z) \pmod{7^4}.
\end{align*}
\end{lemma}
Using these congruences, the $p = 2, 3, 5, 7$ cases of Theorem~\ref{Congruences} immediately follow from Theorem~\ref{thm:Griffin}; the $N = 4, 8, 16, 9, 25$ cases follow by applying the $V_p$ operator to the $p = 2, 3, 5$ cases and using Lemma~\ref{ULemma}.

\bibliographystyle{amsplain}

\end{document}